\documentclass[12pt, reqno]{amsart}

\usepackage{amssymb}
\usepackage{amsmath,amsfonts,amssymb}

\theoremstyle{plain}
\newtheorem{theorem}{Theorem}
\newtheorem{lemma}{Lemma}
\newtheorem{remark}{Remark}[section]
\newtheorem{corollary}{Corollary}[section]
\newtheorem{proposition}{Proposition}[section]
\theoremstyle{proof}
\theoremstyle{definition}

\newtheorem{thma}{Theorem}

\numberwithin{equation}{section}
\numberwithin{lemma}{section}
\numberwithin{theorem}{section}
\usepackage{amsmath}
\usepackage{amsfonts}   
\usepackage{amssymb}
\usepackage{amssymb, amsmath, amsthm}
\usepackage[breaklinks]{hyperref}
\theoremstyle{thmrm}

\usepackage{graphicx}
\begin{document}
\title[On the structure of order $4$ class groups of $\mathbb{Q}(\sqrt{n^2+1})$]{ 
On the structure of order $4$ class groups of $\mathbb{Q}(\sqrt{n^2+1})$}

\author{Kalyan Chakraborty, Azizul Hoque and Mohit Mishra}
\address{Kalyan Chakraborty @Harish-Chandra 
Research Institute, HBNI, Chhatnag Road, Jhunsi, Allahabad 211 019, 
India.}
\email{kalyan@hri.res.in}
\address{Azizul Hoque @Harish-Chandra Research 
Institute, HBNI, Chhatnag Road, Jhunsi, Allahabad 211 019, India.}
\email{ahoque.ms@gmail.com}
\address{Mohit Mishra @Harish-Chandra Research 
Institute, HBNI,
Chhatnag Road, Jhunsi,  Allahabad 211 019, India.}
\email{mohitmishra@hri.res.in}

\maketitle

\begin{abstract} Groups of order $4$ are isomorphic to either 
$\mathbb{Z}/4\mathbb{Z}$ or $\mathbb{Z}/2\mathbb{Z} \times
\mathbb{Z}/2\mathbb{Z}$. We give certain sufficient 
conditions permitting to specify the structure of class groups of order $4$ in
the family  of real quadratic fields  $\mathbb{Q}{(\sqrt{n^2+1})}$ as $n$ varies over positive integers. 
Further, we compute the values of Dedekind zeta function attached to these quadratic fields at the point $-1$. As a side result, we show that  the size of the class group of this family could be made as large  as possible by increasing the size of the number of distinct odd  prime factors of  $n$.
\\
{\bf R\'{e}sum\'{e}.} {\`A } isomorphisme pr\`{e}s, il y a deux   groupes possibles d'ordre $4$: $\mathbb{Z}/4\mathbb{Z}$ et  $\mathbb{Z}/2\mathbb{Z}\times\mathbb {Z}/2\mathbb {Z}$. 
Nous donnons des   conditions suffisantes  permettant de  sp\'{e}cifier la structure des groupes de classes d'ordre $4$ dans  la famille des corps quadratiques r\'{e}els $\mathbb{Q}{(\sqrt{n^2+1})}$  lorsque $n$ parcourt l'ensemble des entiers positifs.
De plus, nous calculons la valeur de la fonction z\^{e}ta de   Dedekind     attach\'{e}e \`{a} ces corps au point $ -1 $. Comme r\'{e}sultat secondaire, nous montrons que la cardinalit\'e du groupe de classes des corps de  cette famille peut \^{e}tre aussi grande que possible en augmentant le nombre de  facteurs premiers impairs distincts de $n$.
\end{abstract}

\section{Introduction}
Let $n$ be a positive integer. We are interested in
the single parametric family of real quadratic fields, 
$
k_n = {\mathbb{Q}}(\sqrt{n^2+1}),
$
which are popularly known as Richaud-Degert (R-D) type real 
quadratic fields. These fields have attracted the attention of many 
mathematicians over the years. Here we refer to few works which are 
connected to our investigation. Chowla and Friedlander \cite{CF76} 
conjectured that if $n^2+1$ is a prime with $n>26$ then the class 
number of $k_n$ is greater than $1$. They also conjectured that if 
$n$ is even and $n^2+1$ is square-free then the class number of 
$k_n$ is greater than $1$ except for $n=2,4,6,10,14, 26$. The first 
conjecture was proved by Mollin and Williams \cite{MW88} under the 
generalized Riemann hypothesis, and the second one was settled by 
Bir\'{o} in \cite{BI03}. Chakraborty and Hoque \cite{CH17} proved
that the class number of $k_n$ is always greater than $1$ if $n=mp
$, where $m$ is an odd integer and $p\equiv \pm 1\pmod 8$ is a 
prime. They also proved an analogous result in \cite{CH18} when 
$n=3m$ with $m$ an odd integer. On the other hand, Yokoi 
\cite{YO86} showed that the class number of $k_n$ is $1$ if and 
only if ${\displaystyle\frac{n^2}{4}-t(t+1)}$ (with ${\displaystyle1\leq t\leq \frac{n}{2}-1}$)  
is a prime. Furthermore, some interesting results on the class number one problem of R-D type (resp. non-R-D type) 
fields were obtained by Bir\'{o} and Lapkova (resp. Hoque and Kotyada) (cf.  \cite{L16, L12, HK20}).

It is of considerable interest to explore the structure of class groups and in 
particular when its order is not  prime. As a starting point, one 
can begin with the classification of class groups of order a prime 
power. One can conclude with the help of Brauer--Siegel theorem
that the class number of $k_n$ can be made as large as possible. A simple version of  
Brauer--Siegel theorem is as follows:  
\begin{thma}[{\cite[Corollary]{SLAN}}]\label{BST}
Let $\Delta_{K,d}$, $R_{K,d}$ and $h_{K, d}$ be respectively, the discriminant, regulator and class number of the number field $K$ of degree $d$ over  $\mathbb{Q}$. If $K$ runs over number fields of degree $d$, then
$$ 
\log(h_{K,d}R_{K, d}) \sim \log (\sqrt{\Delta_{K,d}}).
$$
\end{thma}
Thus there are only finitely many $k_n$ of a given class number $h$.
However, Theorem \ref{BST} is ineffective in finding 
out the exact values of $n$ such that class number of $k_n$ equals 
$h$. On the other hand, Byeon and Kim obtained certain necessary and sufficient 
conditions for the class number of $k_n$ to be $1$ in \cite{BK1} 
(resp. $2$ in \cite{BK2}). 
The authors obtained criteria for  
class number $3$ for $k_n$  in \cite{CHM18} following Byeon-Kim technique. In 
this paper, we give certain sufficient conditions (depending on the 
factorization of $n$) to specify the structure of the class groups of 
$k_n$ of order $4$. Note that 
$4$ is the smallest class number with the 
possibility of two choices of  class groups up to isomorphism. We denote by $\mathfrak{C}(k_n)$ and $\zeta_{k_n}(s)$ respectively,  the class group of $k_n$ and the Dedekind zeta function attached to $k_n$.  The results of this paper are presented below.

We discuss the results based on the class $n^2+1\equiv 1, 2, 5\pmod 8$.  We first consider the case when  $n^2\equiv 0 
\pmod 8$.  Let $n=2^mn_0$ for some integers $m>1$ and $n_0\equiv  
1\pmod 2$. Then $n_0$ has at most two prime factors by 
Proposition \ref{prop1}. 

\begin{theorem}\label{thm1}
Let $d=n^2+1 \equiv 1 \pmod 8$ be a square-free integer, and let 
$h(d)=4$.
\begin{itemize}
\item[(I)] If $n=2^mp^t$ with $m, t\geq 2$ and $p$ an odd prime, 
then \begin{itemize}
\item[(i)] $\mathfrak{C}(k_n)\cong{\displaystyle \mathbb{Z}/4\mathbb{Z}}$,
\item[(ii)] ${\displaystyle \zeta_{k_n}(-1)=\frac{n^3+14n}{360}+
\frac{n^3+n(4p^4+10p^2)}{180p^2}+\frac{n^3+n(4p^8+10p^4)}{360p^4}}.$
\end{itemize} 
\item[(II)] If $n=2^mp^sq^t$ with $p, q$ distinct odd primes, and 
$m>1, s, t\geq 1$ three integers, then
\begin{itemize}
\item[(i)] $\mathfrak{C}(k_n)\cong \mathbb{Z}/2\mathbb{Z}\times 
\mathbb{Z}/2\mathbb{Z}$,
\item[(ii)] $\zeta_{k_n}(-1)={\displaystyle\frac{n^3+32n}{288}+
\frac{n^3+n(4p^4+10p^2)}{360p^2}+\frac{n^3+n(4q^4+10q^2)}{360q^2}}.$
\end{itemize} 
\end{itemize}
\end{theorem}

We now consider  $n^2\equiv 4\pmod 8$. Let $n=2n_0$ for  
 some odd integer $n_0\geq 1$. It follows from Proposition 
\ref{prop1} that $n_0$ has at most three prime factors. 
\begin{theorem}\label{thm2}
Let $d=n^2+1 \equiv 5 \pmod 8$ be a square-free integer, and let
$h(d)=4$.
\begin{itemize}
\item[(I)]
If $n=2p^t$ with $t\geq 3$ an integer, then
\begin{itemize}
\item[(i)] $\mathfrak{C}(k_n)\cong \mathbb{Z}/4\mathbb{Z}$,
\item[(ii)] $\zeta_{k_n}(-1)={\displaystyle\frac{n^3+14n}{360}+
\frac{n^3+n(4p^4+10p^2)}{180p^2}+\frac{n^3+n(4p^8+10p^4)}{360p^4}}.$
\end{itemize} 
\item[(II)] If $n=2p^sq^t$ with $p, q$ distinct odd primes and $s, 
t$ two positive integers such that one of them is greater $2$ or both 
are greater than equal to $2$, then $\mathfrak{C}(k_n)\cong 
\mathbb{Z}/2\mathbb{Z}\times \mathbb{Z}/2\mathbb{Z}$.
\item[(III)] If $n=2p^rq^s\ell^t$ with $p,q,\ell$ distinct odd 
primes and $r, s, t$ positive integers, then
\begin{itemize}
\item[(i)] $\mathfrak{C}(k_n)\cong \mathbb{Z}/2\mathbb{Z}\times 
\mathbb{Z}/2\mathbb{Z}$,
\item[(ii)] $\zeta_{k_n}(-1)={\displaystyle\frac{n^3+14n}{360}+
\frac{n^3+n(4p^4+10p^2)}{360p^2}+\frac{n^3+n(4q^4+10q^2)}{360q^2}+
\frac{n^3+n(4\ell^4+10\ell^2)}{360\ell^2}}.$
\end{itemize}
\end{itemize}
\end{theorem}

Finally  the case $n^2\equiv 1 \pmod 8$ and in this case $n$ is an odd integer. Thus Proposition \ref{prop1} implies that $n$ has at most two prime factors. 

\begin{theorem}\label{thm3}
Let $d=n^2+1 \equiv 2 \pmod 4$ be a square-free integer, and let 
$h(d)=4$.
\begin{itemize}
\item[(I)]
If $n$ is an odd positive integer with two distinct odd prime 
factors $p$ and $q$, then
\begin{itemize}
\item[(i)] $\mathfrak{C}(k_n)\cong \mathbb{Z}/2\mathbb{Z}\times 
\mathbb{Z}/2\mathbb{Z}$,
\item[(ii)] $\zeta_{k_n}(-1)={\displaystyle\frac{n^3+5n}{36}+
\frac{4n^3+n(p^4+10p^2)}{180p^2}+\frac{4n^3+n(q^4+10q^2)}{180q^2}}$.
\end{itemize}
\item[(II)] If $n=p^t$ with $t\geq 2$ an integer and $h(d)=4$, then 
$\mathfrak{C}(k_n)\cong \mathbb{Z}/4\mathbb{Z}$.
\end{itemize}
\end{theorem}

As expected, similar results for 
$d$ of the form $n^2+4$ are also obtained.  
We 
use partial zeta values attached to $k_n$ and generalized Dedekind 
sums which are discussed below. Technically speaking, the method 
adopted here  should work for class groups of order $p^2$ for any 
prime number $p$, but it will be far more challenging to
compute the  zeta values, Dedekind sums, fundamental 
unit and appropriate ideal selection in such a situation.       

\section{Computation of partial Dedekind zeta values}
Let $k$ be a real quadratic field, and $\zeta_{k}(s)$ be the 
Dedekind zeta function attached to $k$. By specializing Siegel's 
formula \cite{SI69} for $\zeta_k(1-2n)$ for general $k$, Zagier 
\cite{ZA76} described this formula by direct analytic methods when 
$k$ is a real quadratic field. This formula takes the following 
shape for $n=1$.
\begin{thma} [{\cite[p. 69]{ZA76}}]\label{thm2.1}
Let $k$ be a real quadratic field with discriminant $D$. Then 
$$
\zeta_k(-1)=\frac{1}{60}\sum_{\substack{ |t|<\sqrt{D}\\ t^2\equiv 
D\pmod 4}}\sigma\left(\frac{D-t^2}{4}\right),
$$
where $\sigma(n)$ denotes the sum of divisors of $n$.\
\end{thma}

Lang \cite{LAN} gave another method to calculate special values of 
$\zeta_k$. We briefly recall Lang's formula. Recall that an 
integral basis for an integral ideal $\mathcal{I}$ of a number 
field $k$ is a set $\{\alpha_1, \alpha_2,\cdots,\alpha_n \}\subset 
\mathcal{I}$, where $n=[k:\mathbb{Q}]$, such that $\mathcal{I}=
\mathbb{Z}\alpha_1 \oplus \mathbb{Z}\alpha_2 \oplus \cdots \oplus 
\mathbb{Z}\alpha_n$. For a given ideal class $\mathfrak{A}$ of $k$, 
let $\mathfrak{a}$ be an integral ideal in $\mathfrak{A}^{-1}$ with 
an integral basis $\{r_{1},r_{2}\}$, i.e. 
$\mathfrak{a}= \mathbb{Z}r_1 \oplus \mathbb{Z}r_2$, where 
$r_1,r_2 \in \mathfrak{a}$. Set $$\delta(\mathfrak{a}):= 
r_1r_2'-r_1'r_2,$$ 
where $r_1'$ and  $r_2'$ are the conjugates of $r_1$ and $r_2$ 
respectively.

Let $\varepsilon$ be the fundamental unit of $k$. Then $           
\{\varepsilon r_1, \varepsilon r_2\}$ is also an integral basis of 
$\mathfrak{a}$, and thus we can find a matrix 
$M=
\begin{bmatrix}
a&b\\
c&d
\end{bmatrix}
$
with integer entries satisfying the following: 
$$
\varepsilon\begin{bmatrix}
r_1\\r_2
\end{bmatrix}
=M\begin{bmatrix}
r_1\\
r_2
\end{bmatrix}.
$$
The following result of Lang helps one to compute partial zeta 
value for $\mathfrak{A}$ at $-1$.

\begin{thma}[{\cite[p. 159]{LAN}}]\label{thm2.2}
By keeping the above notations, we have 
\begin{align*}
\zeta_k(-1, \mathfrak{A})&={\displaystyle\frac{\textsl{sgn }
\delta(\mathfrak{a})~r_2r_2'}{360N(\mathfrak{a})c^3}}\big\{(a
+d)^3-6(a+d)N(\varepsilon)-240c^3(\textsl{sgn } c)\\
&\times S^3(a,c)+180ac^3(\textsl{sgn } 
c)S^2(a,c)-240c^3(\textsl{sgn } c)S^3(d,c)\\
& +180dc^3(\textsl{sgn } c)S^2(d,c) \big\},
\end{align*}
where $N(\mathfrak{a})$ represents the norm of $\mathfrak{a}$ and 
$S^i(-,-)$ denotes the generalized Dedekind sum as defined in 
\cite{AP50}.
\end{thma} 

We need to determine the values of $a,b,c,d$ and generalized 
Dedekind sums in order to apply Theorem \ref{thm2.2}. The following 
result (see, \cite[p. 143, Eq. 2.15]{LAN}) is helpful in determining the 
values of $a, b, c$ and $d$.

\begin{lemma}\label{2.1} The matrix $M$ is given by 
$$
\begin{bmatrix}
Tr\left(\frac{r_1 r_2'\varepsilon}{\delta(\mathfrak{a})}\right)& Tr
\left(\frac{r_1 r_1'\varepsilon'}{\delta(\mathfrak{a})}\right) 
\vspace*{2mm} \\ 
 Tr\left(\frac{r_2 r_2'\varepsilon}{\delta(\mathfrak{a})}\right) &
 Tr\left(\frac{r_1 r_2'\varepsilon'}{\delta(\mathfrak{a})}\right)
\end{bmatrix}
$$
Moreover, $\det(M)=N(\varepsilon)$ and $bc\ne 0$.
\end{lemma} 
The following expressions (see, \cite[pp. 155--157, Eq. 4.3--4.19]
{LAN}) involving special values of  generalized Dedekind sums are 
also needed to compute partial zeta values for ideal classes of a 
real quadratic field. 
\begin{lemma}\label{DS1} For any positive integer $m$, we have
\begin{itemize}
\item[(i)] $S^3(\pm 1, m)={\displaystyle\pm\frac{-m^4+5m^2-4}{120m^3}}.$
\vspace*{2mm}
\item[(ii)] $S^2(\pm 1, m)={\displaystyle\frac{m^4+10m^2-6}{180m^3}}.$
\end{itemize}
\end{lemma}
\begin{lemma}\label{DS2} For any positive even integer $m$, we have
\begin{itemize}
\item[(i)] $S^3(m\pm 1, 2m)=\pm S^1(m+1, 2m)={\displaystyle\mp\frac{m^4-50m^2+4}
{960m^3}}.$ \vspace*{2mm}
\item[(ii)] $S^2(m-1, 2m)=S^2(m+1, 2m)={\displaystyle\frac{m^4+100m^2-6}
{1440m^3}}.$\vspace*{2mm}
\end{itemize}
\end{lemma}
We determine Dedekind zeta values attached to $k_n$ in two ways 
using Theorem \ref{thm2.1} and Theorem \ref{thm2.2} and then 
compare these values and finally use elementary group theoretic 
arguments to establish our results.  

\section{Proof of the results}
Throughout this section $d=n^2+1$ is square-free; $h(d)$ and $
\mathfrak{C}(k_n)$ denote the class number and the class group of 
$k_n=\mathbb{Q}(\sqrt{d})$ respectively. Note that $d\equiv 
1,2,5,6\pmod 8$ in this situation. For us $\mathfrak{P}$ will always refer 
to principal ideal class in the corresponding class group. Let
\begin{align*}
\mathfrak{S}_n:&=\{p\mid n: p \text{ is an odd prime number}\},\\
\mathcal{N}:& =\# \mathfrak{S}_n.
\end{align*}
The following proposition will be needed in the subsequent results. 
\begin{proposition}\label{prop1}
If $\mathcal{N}\geq 3$ then 
$$h(d) \geq \begin{cases} \mathcal{N}+2 & \text{ if }  d \equiv 
1,2,6 \pmod 8, \\
\mathcal{N}+1 & \text{ if }  d \equiv 5 \pmod 8. \\
\end{cases}$$
%Moreover, equality holds if and only if 
%$$
%\zeta_{k_n}(-1)=\begin{cases} \frac{n^3+32n}{288} + \sum
%\limits_{p_i\in \mathfrak{S}_n} \frac{n^3+n(4p_{i}^4+10p_{i}^2)}
%{360p_{i}^2} & \text{ if }  d \equiv 1 \pmod 8, \\
%\frac{n^3+5n}{36} + \sum\limits_{p_i\in \mathfrak{S}_n} 
%\frac{4n^3+n(p_{i}^4+10p_{i}^2)}{180p_{i}^2} & \text{ if }  d 
%\equiv 2,6 \pmod 8, \\
%\frac{n^3+14n}{360} + \sum\limits_{p_i\in \mathfrak{S}_n} 
%\frac{n^3+n(4p_{i}^4+10p_{i}^2)}{360p_{i}^2} & \text{ if }  d 
%\equiv 5 \pmod 8.\\
%\end{cases}
%$$ 
\end{proposition}
\begin{proof}
We will provide the complete proof for the case $d \equiv 1 \pmod 
8$ and other cases can be handled along the same lines.  In this case 
$2$ splits in  $k_n$ as:
\begin{equation}\label{eq2}
(2)=\left(2,\frac{1+\sqrt{d}}{2}\right)\left(2,\frac{1-\sqrt{d}}{2}
\right).
\end{equation}
Also $p_i\in\mathfrak{S}_n$ splits in $k_n$ as:
\begin{equation}\label{eq1}
(p_i)=\left(p_i,\frac{1+\sqrt{d}}{2}\right)\left(p_i,\frac{1-
\sqrt{d}}{2}\right).
\end{equation}
Let $\mathfrak{A}_i$ and $\mathfrak{B}$ be two ideal classes in 
$k_n$ such that $\mathfrak{a}_i=\left(p_i,\frac{1+\sqrt{d}}{2}
\right) \in \mathfrak{A}_i$ and $\mathfrak{b}\left(2,\frac{1+
\sqrt{d}}{2}\right) \in \mathfrak{B}$. Then $\mathfrak{a}_i^{-1}=
\left(p_i,\frac{1-\sqrt{d}}{2}\right) \in \mathfrak{A}_i^{-1}$ and 
$\mathfrak{b}^{-1}=\left(2,\frac{1-\sqrt{d}}{2}\right) \in 
\mathfrak{B}^{-1}$. Consider a non-zero $\mathbb{Z}$-module $M_i=
\Big[p_i,\frac{1-\sqrt{d}}{2}\Big]$ in $\mathcal{O}_{k_n}$. 
Then, by \cite[Propositions 2.6 and 2.11]{Fra}, $M_{i}$ is an 
integral ideal and $N(M_{i})=p_{i}$. Also $N(\mathfrak{a}_{i}
^{-1})=p_i$ and $M_i \subseteq \mathfrak{a}_i^{-1}$, therefore \
$M_i=\mathfrak{a}_i^{-1}$. Hence $\{p_i,\frac{1-\sqrt{d}}{2}\}$  is 
an integral basis for $\mathfrak{a}_i^{-1}$, i.e. $\mathfrak{a}
_i^{-1}=\mathbb{Z}p_i \oplus \mathbb{Z}\frac{(1-\sqrt{d})}{2}$. 
Similarly, 
$\{2,\frac{(1-\sqrt{d})}{2}\}$ is an integral basis for 
$\mathfrak{b}^{-1}$, i.e. $\mathfrak{b}^{-1}= \mathbb{Z}2 \oplus 
\mathbb{Z} \frac{(1-\sqrt{d})}{2}$.  Now using Lemma \ref{2.1}, 
Lemma \ref{DS1} and Theorem \ref{thm2.2}, we obtain  
$$\zeta_{k_n}(-1,\mathfrak{A}_{i})=\frac{n^3+n(4p_{i}^4+10p_{i}^2)}
{360p_{i}^2},$$ and $$\zeta_{k_n}(-1,\mathfrak{B})=\frac{n^3+104n}
{1440}.$$
Also by \cite[Theorem 2.3]{BK1},
\begin{equation}\label{eq3}
\zeta_{k_n}(-1, \mathfrak{P})=\frac{n^3+14n}{360}.
\end{equation}
If $\zeta_{k_n}(-1,\mathfrak{A}_{i})=\zeta_{k_n}(-1, \mathfrak{P})
$, then we must have $n=2p_i$. This contradicts the fact that $
\mathcal{N}\geq 3$. Similarly, $\zeta_{k_n}(-1,\mathfrak{B})=
\zeta_{k_n}(-1, \mathfrak{P})$ gives $n=4$, which is again a 
contradiction. Also, $\zeta_{k_n}(-1,\mathfrak{B})=\zeta_{k_n}(-1,
\mathfrak{A}_{i})$ implies $n=4p_i$, which is not possible. Finally 
if $\zeta_{k_n}(-1,\mathfrak{A}_{i})=\zeta_{k_n}(-1,\mathfrak{A}
_{j})$ for $i \neq j$, then $n=2p_ip_j$. This again contradicts the 
fact 
that $\mathcal{N}\geq 3$.
 
Therefore $\mathfrak{A}_{i}~ (1\leq i\leq \mathcal{N})$ and $
\mathcal{B}$ are distinct non-principal ideal classes in $k_n$, and 
thus $h(d) \geq \mathcal{N}+2$. 
%\\
%Now by definition, we have
%$$
%\zeta_{k_n}(-1)=\sum_{\mathfrak{I}\in \mathfrak{C}(k_n)}\zeta_{k_n}
%(-1, \mathfrak{I}).
%$$
%Since, for any $\mathfrak{I} \in \mathfrak{C}(k_n)$, 
%$$
%\zeta_{k_n}(s, \mathfrak{I})=\sum_{\mathfrak{a}\in \mathfrak{A}}
%\frac{1}{\left(N\mathfrak{a}\right)^s}.
%$$ 
%Therefore $\zeta_{k_n}(-1, \mathfrak{A}) > 0$ and we obtain (since 
%$\{\mathfrak{P}, \mathfrak{B}, \mathfrak{A}_1, \cdots \mathfrak{A}
%_{\mathcal{N}}\}\subseteq \mathfrak{C}(k_n)$)
%$$\zeta_{k_n}(-1) \geq \zeta_{k_n}(-1,\mathfrak{P})+\zeta_{k_n}(-1,
%\mathfrak{B})+\sum\limits_{i=1}^{\mathcal{N}} \zeta_{k_n}(-1,
%\mathfrak{A}_{i}).$$
%This implies that
%$$\zeta_{k_n}(-1) \geq  \frac{n^3+32n}{288} + \sum\limits_{i=1}
%^{\mathcal{N}} \frac{n^3+n(4p_{i}^4+10p_{i}^2)}{360p_{i}^2}.$$
%This further shows that equality holds if and only if $h(d)=
%\mathcal{N}+2$.
\end{proof}
As a consequence, we obtain the following interesting result.
\begin{corollary} $h(d) \to \infty$ as $\mathcal{N} \to \infty$.
\end{corollary}

\subsection*{Proof of Theorem \ref{thm1}} 
We give the proof of the first part in details, and then we give 
the outline of the second part.
In the case $ n=2^mp^t$,  
both $2$ and $p$ split in $k_n$ as in \eqref{eq2} and \eqref{eq1} 
respectively.

Let $\mathfrak{A}$ be the ideal class in $k_n$ such that $
\mathfrak{a}=\left(p,\frac{1+\sqrt{d}}{2}\right) \in \mathfrak{A}$. 
By multiplication formula for ideals in quadratic fields \cite[p. 
48]{MOL}, we will have $\left(p,\frac{1+\sqrt{d}}{2}\right)^2=
\left(p^2,\frac{1+
\sqrt{d}}{2}\right)$ and $\left(p,\frac{1-\sqrt{d}}{2}\right)^2=
\left(p^2,\frac{1-\sqrt{d}}{2}\right)$. Then $\mathfrak{a}^{-1}=
\left(p,\frac{1-\sqrt{d}}{2}\right) \in \mathfrak{A}^{-1}$, $
\mathfrak{a}^2=\left(p^2,\frac{1+\sqrt{d}}{2}\right) \in 
\mathfrak{A}^2$ and $(\mathfrak{a}^2)^{-1}=\left(p^2,\frac{1-
\sqrt{d}}{2}\right) \in (\mathfrak{A}^2)^{-1}$.
Consider nonzero $\mathbb{Z}$-modules $M_{r}=\Big[p^r,
\frac{1-\sqrt{d}}{2}\Big]$ in $\mathcal{O}_{k_n}$, where $r \in \{1,2\}$. Then again by \cite[Propositions 2.6 and 2.11]{Fra}, $M_{r}$ is an ideal and $N(M_{r})=p^r$, for all $r \in \{1,2\}$. Also $N((\mathfrak{a}^r)^{-1})=p^r$ and $M_{r} \subseteq (\mathfrak{a}^r)^{-1}$, for all $r \in \{1,2\}$. Therefore $M_{r} = 
(\mathfrak{a}^r)^{-1}$, for all $r \in \{1,2\}$. Hence $\{ p,\frac{1-\sqrt{d}}{2}  \}$ and $\{ p^2,\frac{1-\sqrt{d}}{2} \}$ are an integral basis for $\mathfrak{a}^{-1}$ 
and $(\mathfrak{a}^2)^{-1}$ respectively. Now using Lemma 
\ref{2.1}, Lemma \ref{DS1} and Theorem \ref{thm2.2}, we obtain $$
\zeta_{k_n}(-1,\mathfrak{A})=\frac{n^3+n(4p^4+10p^2)}{360p^2},$$ 
and $$\zeta_{k_n}(-1,\mathfrak{A}^2)=\frac{n^3+n(4p^8+10p^4)}
{360p^4}.$$
Also $\zeta_{k_n}(-1, \mathfrak{P})$ is given by \eqref{eq3}.
We claim that $\mathfrak{A}$ is a generator of $\mathfrak{C}(k_n)$. 
To see this we pairwise equate $\zeta_{k_n}(-1, \mathfrak{P})$, $
\zeta_{k_n}(-1, \mathfrak{A})$ and $\zeta_{k_n}(-1, \mathfrak{A}^2)
$ to find that $n=2p, 2p^2$. These values are not possible since 
$m, t\geq 2.$ Thus order of $\mathfrak{A}$ is greater than $2$. 
This establishes our claim as $h(d)=4$.
Therefore $\mathfrak{C}(k_n)\cong \mathbb{Z}/4\mathbb{Z}$. 

Now 
$$\zeta_{k_n}(-1)=\zeta_{k_n}(-1, \mathfrak{P})+\zeta_{k_n}(-1, 
\mathfrak{A})+\zeta_{k_n}(-1, \mathfrak{A}^2)+\zeta_{k_n}(-1, 
\mathfrak{A}^{-1}).$$
Thus,
$$\zeta_{k_n}(-1)=\frac{n^3+14n}{360}+\frac{n^3+n(4p^4+10p^2)}
{180p^2}+\frac{n^3+n(4p^8+10p^4)}{360p^4}.$$
This completes the proof of (I) of Theorem \ref{thm1}.

In the case $n=2^mp^sq^t$, both $p$ 
and $q$ split as in \eqref{eq1}, and $2$ splits as in \eqref{eq2}. 
Let $\mathfrak{A}, ~ \mathfrak{B}$ and $\mathfrak{C}$ be the three 
ideal classes in $k_n$ such that $\left(p,\frac{1+\sqrt{d}}{2}
\right) \in \mathfrak{A}$, $\left(q,\frac{1+\sqrt{d}}{2}\right) \in 
\mathfrak{B}$ and $\left(2,\frac{1+\sqrt{d}}{2}\right) \in 
\mathfrak{C}$. Then analogous to the previous case, we obtain
\begin{align*}
\zeta_{k_n}(-1,\mathfrak{A})&=\frac{n^3+n(4p^4+10p^2)}{360p^2},\\
\zeta_{k_n}(-1,\mathfrak{B})&=\frac{n^3+n(4q^4+10q^2)}{360q^2},\\
\zeta_{k_n}(-1,\mathfrak{C})&=\frac{n^3+104n}{1440}.
\end{align*}
Equating pairwise these values as before, we get $$n=4, 2p, 2q, 
2pq, 4p, 4q.$$  Since $h(d)=4$, $2^2|n$ and $s, t\geq 1$, these 
values of $n$ do not arise. Hence $\mathfrak{A}, ~ \mathfrak{B}$ 
and $\mathfrak{C}$ are distinct non-principal ideal classes in $k_n
$.  
Using these three non-principal ideal classes, we can prove (II) of Theorem \ref{thm1} following  similar arguments.

\subsection*{Proof of Theorem \ref{thm2}} 
We first consider the case $n=2p^t$ for some integer $t\geq 3$, and thus 
$p$ splits in $k_n$ similar to \eqref{eq1}. 

Let $\mathfrak{A}$ be an ideal class containing $\mathfrak{a}=
\left(p,\frac{1+\sqrt{d}}{2}\right)$. Then again by multiplication 
formula for ideals in quadratic fields \cite[p. 
48]{MOL}, $\mathfrak{a}^2=\left(p,\frac{1+\sqrt{d}}{2}\right)^2= 
\left(p^2,\frac{1+\sqrt{d}}{2}\right)$  and $(\mathfrak{a}^2)^{-1}=
\left(p,\frac{1-\sqrt{d}}{2}\right)^2= \left(p^2,\frac{1-\sqrt{d}}
{2}\right)$. Thus $\mathfrak{a}^2= \left(p^2,\frac{1+\sqrt{d}}{2}
\right) \in \mathfrak{A}^2$, $\mathfrak{a}^{-1}=\left(p,\frac{1-
\sqrt{d}}{2}\right) \in \mathfrak{A}^{-1}$ and $(\mathfrak{a}
^2)^{-1}= \left(p^2,\frac{1-\sqrt{d}}{2}\right) \in (\mathfrak{A}
^2)^{-1}$. As proved in the proof of Theorem \ref{thm1}, $\{ p,\frac{1-\sqrt{d}}
{2}  \}$ and $\{ p^2,\frac{1-\sqrt{d}}{2} \}$ are integral basis 
for $\mathfrak{a}^{-1}$ and $(\mathfrak{a}^2)^{-1}$ respectively. 
Now by using Lemma \ref{2.1}, Lemma \ref{DS1} and Theorem 
\ref{thm2.2} we obtain:
\begin{align*}
\zeta_{k_n}(-1, \mathfrak{A})&=\frac{n^3+n(4p^4+10p^2)}{360p^2},\\
\zeta_{k_n}(-1, \mathfrak{A}^2)&=\frac{n^3+n(4p^8+10p^4)}{360p^4}.
\end{align*}
As in \S2.1, we observe that $\mathfrak{A}$ and $\mathfrak{A}^2$ 
are distinct non-principal ideals classes in $k_n$ since $t\geq 3$. 
This implies that the order of $\mathfrak{A}$ in $\mathfrak{C}(k_n)
$ is greater than $2$. Therefore the order of $\mathfrak{A}$ in $
\mathfrak{C}(k_n)$ is $4$ and hence $\mathfrak{C}(k_n)\cong 
\mathbb{Z}/4\mathbb{Z}$.

Since $\mathfrak{A}$ is a generator and $\mathfrak{P}$ is principal 
in $\mathfrak{C}(k_n)$, we have
$$\zeta_{k_n}(-1)=\zeta_{k_n}(-1, \mathfrak{P})+\zeta_{k_n}(-1, 
\mathfrak{A})+\zeta_{k_n}(-1, \mathfrak{A}^2)+\zeta_{k_n}(-1, 
\mathfrak{A}^{-1}).$$
Thus,
$$\zeta_{k_n}(-1)=\frac{n^3+14n}{360}+\frac{n^3+n(4p^4+10p^2)}
{180p^2}+\frac{n^3+n(4p^8+10p^4)}{360p^4}.$$
This completes the proof of (I) of Theorem \ref{thm2}.

In the next case, both $p$ and $q$ split 
as in \eqref{eq1}. 
Let $\mathfrak{A}$ and $\mathfrak{B}$ are two ideal classes in $k_n
$ such that $\left(p,\frac{1+\sqrt{d}}{2}\right) \in \mathfrak{A}$ 
and $\left(q,\frac{1+\sqrt{d}}{2}\right) \in \mathfrak{B}$. As in 
\S2.1, we see that both $\mathfrak{A}$ and $\mathfrak{B}$ are 
distinct non-principal in $\mathfrak{C}(k_n)$ since either $s>2$ or 
$t>2$.

As $\left(p,\frac{1-\sqrt{d}}{2}\right) \in \mathfrak{A}^{-1}$ and 
$\left(q,\frac{1-\sqrt{d}}{2}\right) \in \mathfrak{B}^{-1}$, we 
have
\begin{align*}
\zeta_{k_n}(-1,\mathfrak{A}^{-1})=\frac{n^3+n(4p^4+10p^2)}{360p^2},
\\
\zeta_{k_n}(-1,\mathfrak{B}^{-1})=\frac{n^3+n(4q^4+10q^2)}{360q^2}.
\end{align*}
Also $\zeta_{k_n}(-1,\mathfrak{A}^{-1})=\zeta_{k_n}(-1,
\mathfrak{A})$
and $\zeta_{k_n}(-1,\mathfrak{B}^{-1})=\zeta_{k_n}(-1
\mathfrak{B})$. If $\mathfrak{C}(k_n)\cong \mathbb{Z}/4\mathbb{Z}
$, then either $\mathfrak{A}$ or $\mathfrak{B}$ is a generator.
Let it be $\mathfrak{A}$. Then $\mathfrak{B}\neq \mathfrak{A}^{-1}
$, as the corresponding partial Dedekind zeta values are not equal. 
This forces that $\mathfrak{A}^2=\mathfrak{B}$, and therefore $
\zeta_{k_n}(-1,\mathfrak{A}^2)=\zeta_{k_n}(-1,\mathfrak{B})$, i.e.
$$
\frac{n^3+n(4p^8+10p^4)}{360p^4}=\frac{n^3+n(4q^4+10q^2)}{360q^2}.
$$ 
On simplification, we get $n=2p^2q$, which contradicts the
assumption on $n$, and hence contradiction to our assumption that $
\mathfrak{C}(k_n)\cong \mathbb{Z}/4\mathbb{Z}
$. 
Therefore the only possibility is that, the 
order of each of $\mathfrak{A}$ and $\mathfrak{B}$ must be $2$. 
This completes the proof of (II) of Theorem \ref{thm2}.

In the last case, that is when $n=2p^rq^s \ell^t$, we see that $p,q$ and $
\ell$ split in $k_n$ as in \eqref{eq1}.
If $\mathfrak{A}, ~ \mathfrak{B}$ and $\mathfrak{C}$ are ideal 
classes in $k_n$ containing $\left(p,\frac{1+\sqrt{d}}{2}\right),~
\left(q,\frac{1+\sqrt{d}}{2}\right)$ and $\left(\ell,\frac{1+
\sqrt{d}}{2}\right)$ respectively, then similar to the previous 
cases,
\begin{align*}
\zeta_{k_n}(-1,\mathfrak{A})=\frac{n^3+n(4p^4+10p^2)}{360p^2},\\
\zeta_{k_n}(-1,\mathfrak{B})=\frac{n^3+n(4q^4+10q^2)}{360q^2},\\
\zeta_{k_n}(-1,\mathfrak{C})=\frac{n^3+n(4\ell^4+10\ell^2)}
{360\ell^2}.
\end{align*}
Proceeding as before one can show that these ideal classes are 
distinct and non-principal in $k_n$. Using these ideal classes, we can prove (III) of Theorem \ref{thm2}.

\subsection*{Proof of Theorem \ref{thm3}} Recall that $n$ has at most two prime factors. Also, 
 the zeta value at $-1$ for $\mathfrak{P}$ is given by 
\cite[Theorem 2.3]{BK1}, i.e.,
$$\zeta_{k_n}(-1, \mathfrak{P})=\frac{4n^3+11n}{180}.$$
We give the proof of (I) of Theorem \ref{thm3}, and the analogous arguments work for the remaining part. 

Since $p, q\mid n$, so that  both $p$ and $q$ split in $k_n$, that is,
\begin{align}\label{eqx}
\begin{cases}
(p)=\left(p,1+\sqrt{d}\right)\left(p,1-\sqrt{d}\right), 
\vspace*{2mm}\\ 
(q)=\left(q,1+\sqrt{d}\right)\left(q,1-\sqrt{d}\right).
\end{cases}
\end{align}
Also $(2)=(2,\sqrt{d})^2$. Let $\mathfrak{A}, ~ \mathfrak{B}$ and $
\mathfrak{C}$ be ideal classes in $k_n$ such that $(2,\sqrt{d})\in 
\mathfrak{A}$, $(p,1+\sqrt{d})\in \mathfrak{B}$ and $
(q,1+\sqrt{d})\in\mathfrak{C}$. 
Now if we consider non-zero $\mathbb{Z}$-modules $N=[p,1-\sqrt{d}]$ and $N^{\prime}=[q,1-\sqrt{d}]$ in $\mathcal{O}_{k_n}$, then as before (using \cite[Proposition 2.6 and 2.11]{Fra})  one can prove that $N=(p,1-\sqrt{d})$  and $N^{\prime}=(q,1-\sqrt{d})$. Hence, $\{p,1-\sqrt{d}\}$  and $\{q,1-\sqrt{d}\}$ are an 
integral basis for $(p,1-\sqrt{d})$ and $(q,1-\sqrt{d})$ respectively.
Now as before:
\begin{align*}
\zeta_{k_n}(1, \mathfrak{A}) & =\frac{n^3+14n}{180},\\
\zeta_{k_n}(-1, \mathfrak{B}) & =\frac{4n^3+n(p^4+10p^2)}{180p^2},
\\
\zeta_{k_n}(-1, \mathfrak{C}) & =\frac{4n^3+n(q^4+10q^2)}{180q^2}.
\end{align*}
Employing similar technique, we see that all these three ideal 
classes are distinct and  non-principal in $k_n$.

Clearly $\mathfrak{C}$ is of order $2$. We claim that both $
\mathfrak{A}$ and $\mathfrak{B}$ are of order $2$. We observe from 
\eqref{eqx} that $(p,1-\sqrt{d})\in \mathfrak{A}^{-1}$ 
and $(q,1-\sqrt{d})\in \mathfrak{B}^{-1}$. We see that $
\zeta_{k_n}(-1, \mathfrak{A})=\zeta_{k_n}(-1, \mathfrak{A}^{-1})$ 
and $\zeta_{k_n}(-1, \mathfrak{B})=\zeta_{k_n}(-1, \mathfrak{B}
^{-1})$. Also $\zeta_{k_n}(-1, \mathfrak{A}^{-1}) \neq \zeta_{k_n}
(-1, \mathfrak{B}^{-1}) \neq \zeta_{k_n}(-1, \mathfrak{C})$ and 
hence our claim is proved.  Therefore $\mathfrak{C}(k_n)\cong 
\mathbb{Z}/2\mathbb{Z} \times \mathbb{Z}/2\mathbb{Z}$.

Also 
$$
\zeta_{k_n}(-1)=\zeta_{k_n}(-1, \mathfrak{P})+\zeta_{k_n}(-1, 
\mathfrak{A})+\zeta_{k_n}(-1, \mathfrak{B})+\zeta_{k_n}(-1, 
\mathfrak{C}).
$$
Therefore
$$
\zeta_{k_n}(-1)=\frac{n^3+5n}{36}+\frac{4n^3+n(p^4+10p^2)}{180p^2}+
\frac{4n^3+n(q^4+10q^2)}{180q^2}.
$$
This completes the proof of (I) of Theorem \ref{thm3}. 

In the case when $n$ is power of a single prime, i.e., $n=p^t$ for some 
integer $t\geq 3$, $p$ splits as in the last case.
Following the previous arguments, one can prove (II) of Theorem \ref{thm3}.

\section{Concluding remarks}
Real quadratic fields of the form $k_{n,r}=\mathbb{Q}(\sqrt{n^2+r})
$ (with $n^2+r$ square-free) are called Richaud-Degert(R-D) type if 
$r|4n$ for two integers $n$ and $r$. Furthermore,  if $|r|=1,4$, then 
$k_{n,r}$ is known as narrow R-D type; otherwise $k_{n,r}$ is known 
as extended R-D type field. We have obtained sufficient conditions 
to specify the structure of order $4$ class groups in $k_{n,r}$ when 
$r=1$. Following the same method, we can  obtain sufficient 
conditions to specify order $4$ class groups of 
$k_{n,r}$ when $r=4$. More precisely:
\begin{theorem} 
Let $d=n^2+4$ be a square-free integer, and let 
$h(d)=4$. The following statements hold:
\begin{itemize}
\item[(i)] If $n=p^t$ with $t \geq 2$, then $\mathfrak{C}(k_{n,4})
\cong \mathbb{Z}/4\mathbb{Z}$.
\item[(ii)] If $n=p^rq^s$ with $r\geq 2 $ or $s \geq 2$
then $\mathfrak{C}(k_{n,4}) \cong \mathbb{Z}/2\mathbb{Z}\times 
\mathbb{Z}/2\mathbb{Z}$.
\item[(iii)] If $n$ has more than two distinct odd prime factors, 
then $\mathfrak{C}(k_{n,4}) \cong \mathbb{Z}/2\mathbb{Z}\times 
\mathbb{Z}/2\mathbb{Z}$.
\end{itemize}
\end{theorem}

\begin{remark}
The method followed here may not work for 
$r=-1$ and $-4$.  The ideals used here will not be helpful and  it will 
not be easy to calculate generalized Dedekind sums for other 
ideals. However, it would be interesting to extend our method to 
extended R-D type real quadratic fields. 
\end{remark}

\section*{Acknowledgements} 
\noindent 
The authors would like to express their gratitude to Professor Claude Levesque for 
carefully reading this manuscript and for his useful comments. The second author is grateful to Professor Srinivas Kotyada for stimulating environment at The Institute of Mathematical Sciences, Chennai during his visiting period. The 
authors are thankful to the anonymous referees for their valuable  
comments and suggestions which have helped improving the presentation immensely. The second author acknowledges the grant SERB MATRICS Project (No. MTR/2017/00100). The third author is partially supported by 
`Infosys grant'.
%%%%%%%%%%%%%%%%%%%%%%%%%%%%%%%%%%%%%%%%%%%%%%%%%%%      

\end{document}